\documentclass{article}
\usepackage{graphicx} 
\usepackage{amsmath}
\usepackage{amsthm}
\usepackage{amsfonts}
\usepackage{amssymb}
\usepackage{natbib} 
\usepackage{extarrows}
\usepackage{framed}
\usepackage{enumerate} 
\usepackage{graphicx} 
\usepackage{geometry}
\usepackage{hyperref}  

\hypersetup{
    colorlinks=true,   
    linkcolor=blue,    
    citecolor=blue,    
    urlcolor=blue      
}

\numberwithin{equation}{section}

\theoremstyle{plain}
\newtheorem{thm}{\bf Theorem}[section]

\newtheorem{lem}[thm]{\bf Lemma}

\newtheorem{prop}[thm]{\bf Proposition}
\newtheorem{defi}[thm]{\bf Definition}
\newtheorem{ex}[thm]{\bf Example}

\newcommand*{\re}{\mathbf{R}}

\newcommand*{\ene}{\mathbf{e}}

\newcommand*{\zz}{\mathbf{Z}}
\newcommand*{\nat}{\mathbf{N}}

\newcommand*{\ff}{\mathcal{F}}

\newcommand{\Rmnum}[1]{\uppercase\expandafter{\romannumeral #1}}

\date{}
\begin{document}

    \title{\Large \bf \boldmath\ \\ ESTIMATION OF FIRST RETURNING SPEED IN NULL RECURRENT CONTINUOUS-TIME MARKOV CHAINS}

     \author{\large Huixi WANG$^1$ \qquad Minzhi ZHAO$^2$ }
 
    \maketitle

    \renewcommand{\thefootnote}{\fnsymbol{footnote}}
    \footnotetext{\hspace*{-5mm} \begin{tabular}{@{}r@{}p{13.4cm}@{}}
    & Manuscript received  \\ 
    $^1$ & Department of Mathematics, Fudan University,
    Shanghai 200433, China.\\
    &{E-mail:hxwang20@fudan.edu.cn} \\
    $^{2}$ & Department of Mathematics, Zhejiang University,
    Hangzhou, Zhejiang 310058, China.\\
    &{E-mail:zhaomz@zju.edu.cn} \\
	\end{tabular}}

\begin{abstract}
    This paper establishes a novel connection between null-recurrent CTMCs and electric networks, offering a systematic classification of null-recurrent behavior based on the first returning speed. By leveraging techniques from electric network theory, we present a general method for estimating the first returning speed of null recurrent birth-death processes and provide some important examples.
\end{abstract}

\vskip 4.5mm
\begin{tabular}{@{}l@{ }p{10.1cm}} {\bf Keywords } &
Null recurrence, First returning time, Continuous-time Markov chains, Electric network, Birth-death processes
\end{tabular}

{\bf 2020 MR Subject Classification } 60J27

\baselineskip 14pt

\setlength{\parindent}{1.5em}

\section{Introduction}
Markov chains are a fundamental class of stochastic processes that plays a crucial role in various fields. In particular, Continuous-Time Markov Chains (CTMCs) on a countable state space are widely used in applications, for example, in genetics \cite{ewens2004mathematical}, cell biology \cite{bressloff2014stochastic}, epidemiology \cite{pastor2015epidemic}, queueing theory \cite{kalashnikov2013mathematical} and many other fields.

In some literature, CTMCs are also referred to as jump processes\cite{fukushima2011dirichlet} or Q-processes\cite{chen2004markov}. In the study of CTMCs, the two most important problems are the construction and the criteria for its uniqueness, and the dynamical properties (explosivity\cite{reuter1957denumerable}, recurrence, transience, ergodicity, etc.) of CTMCs. Over the past 50 years, a substantial amount of research has been conducted on these topics\citep{Wang1992Birth, anderson2012continuous,norris1998markov}. Using the moments of the last exit time\cite{hu2019moments}, first hitting time (or first returning time)\cite{gong2012hitting} to a further characterization is a problem of widespread interest. These concepts play crucial roles in probability theory and stochastic analysis, enabling deeper understanding of the dynamic behaviors and statistical properties of random processes. These results are also widely applied in various models\citep{xu2023full,backenkohler2020bounding,narayanan1996first}.

\subsection{Main results}
The main results of this paper are as follows. For null recurrent CTMCs, we first provide 
the estimation of the tail probability distribution for the first returning time $\tau_x$. Then we derive some equivalent characterizations of estimating \( E^x(w(\tau_x))\), where $w(\cdot)$ is a weight function on $[0,+\infty)$ that grows more slowly than linearly and needs to satisfy some technical conditions. Subsequently, inspired by \cite{Zhao2006null}, we establish an electric network for CTMCs and leverage the techniques of electric networks to present an equivalent characterization of \( E^x(w(\tau_x))\). By applying the electric network approach, we can precisely compute the necessary and sufficient conditions for \( E^x(w(\tau_x))<\infty\) with respect to any given weight function \( w(\cdot) \) in birth-death processes. Using these theorems, we can classify numerous important null recurrent CTMCs, particularly birth-death processes, based on the established framework.

\subsection{Comparison with results in the literature and further extensions}

Regarding CTMCs (esp. birth-death processes) with a countable state space $I$, classical research on the first passage time or first returning time typically starts from the perspective of solving equations 
\[\lambda f-Qf=0.\]
In the above equation, $f(\cdot)$ is a bounded function defined on $I$, $\lambda>0$ and $Q$ is the density matrix. We can construct the resolvent by using the solutions of equations, so as to represent the Laplace transform of first passage times $E^x(e^{-\lambda T_y})$. 
Regarding the application of this important method, numerous references can be consulted, such as \cite{Wang1992Birth,chen2004markov}. This approach has given rise to many analytical techniques, such as the martingale method\cite{menshikov2014explosion,kou2003first}, approximations by other CTMCs\cite{al2015moment,horvath2024approximation}. These methods mainly focus on constructing the behavior of the process at boundaries or proving the existence of moments of first passage times, but they are often difficult to compute for specific processes. Additionally, combinatorial techniques were used in some studies to derive exact solutions for certain special processes\cite{szewczak2008moments}, though this approach lacks generality. Meanwhile, there are some results on the moments of first returning times for positive recurrent discrete Markov chains, which can be referred to\cite{aurzada2011moments,chung1967markov}. Some results on null-recurrent discrete-time Markov chains (DTMCs) can be referred to \cite{Zhao2006null}, which gives a characterization of the electric networks corresponding to null-recurrent DTMCs and some basic properties.

This paper mainly focuses on the characterization of the first returning time for null-recurrent CTMCs. From the perspective of electric networks, it provides a general computational framework for characterizing \( E^x(w(\tau_x))\) of birth-death processes. For general birth-death processes, such strict estimation is difficult to compute using other methods.

\subsection{Outline}

The paper is organized as follows. In Section 2, the basic notations and some preliminaries we used about CTMC are introduced. The estimation of the tail probability distribution for the first returning time $\tau_x$ and some equivalent characterizations of \( E^x(w(\tau_x))<\infty\) are derived in Section 3. In Section 4, electric networks for CTMC and some techniques of electric networks are provided. An important characterization of \( E^x(w(\tau_x))<\infty\) in electric networks is proved as well. In Section 5, we provide a standard procedure for estimating \( E^x(w(\tau_x))\) and determine the results applied in the linear growth model and bilateral birth-death process. 

\section{Preliminaries on continuous-time Markov chains}
A continuous-time Markov chain is a Markov process $\{X_t,t\geq0\}$ on $(\Omega,\ff,P)$ that evolves in continuous time and has a countable state space $I$. Assume that $P(t,x,y)(x,y\in I)$ is the transition probability function or matrix on $[0,\infty)\times I\times I,$ which means \[P(X_t=y|X_0=x)=P(t,x,y).\] Sometimes, when there is no ambiguity, we will also write $P(t,x,y)$ as $P_{xy}(t)$. In this paper, we always assume that the probability matrix $P_{xy}(t)$ is {\bf honest} and {\bf standard}, that is to say \[\sum _{y\in I} P_{xy}(t)=1,\,\,\,\lim_{t\downarrow 0}P_{xy}(t)=\delta_{xy}, \]
for any $x,y\in I$ and $t\geq 0$.

Let $T_{y}=\inf\{t > 0: X_t = y\}$ be {\bf the first hitting time} of state $y$. Notice that, if the process $X_t$ starts from $y$, then $P^y(T_{y}=0)=1$. Let {\bf the first discontinuous point} be $\eta_1:=\inf \{t>0:X_t\neq X_0\}$. {\bf The first returning time} of state $x$ is denoted by $\tau_x=\inf\{t>\eta_1:X_t=x\}$. It's obviously that $P^x(\tau_y=T_y)=1$ for $x\ne y$. According to Theorem \Rmnum{2}.15.4 in \cite{chung1967markov}, both $T_{y}$ and $\tau_x$ have continuous probability density functions. 

The transition rates of a CTMC are often described using a density matrix denoted by $Q=(q_{xy})$. Throughout this paper, we consistently assume that the Q-matrix $Q=(q_{xy})$ is {\bf regular}. 
For a given transition probability matrix $P(t,x,y)$, a regular Q-matrix $Q=(q_{xy})$ means  
\[q_{xy}=\lim_{t\downarrow 0} \dfrac{P_{xy}(t)-\delta_{xy}}{t},\,\,x,y\in I\]
where $0\leq q_{xy}< \infty$ for $x\neq y$ and $0< q_x:=-q_{xx}< \infty,\, \sum_{y\in I}q_{xy} = 0.$ However, given a Q-matrix, the process is not necessarily unique. We have already defined $\eta_1$ the first discontinuous point of $\{X_t,t\geq0\}$. By the definition of density matrix, \[P^x(\eta_1>t)=e^{-q_{x}t},\] which means the first discontinuous point of $\{X_t,t\geq0\}$ exists. We denote $\eta_2$ the first discontinuous point after $\eta_1$. Then $\eta_2$ is the second discontinuous point of $\{X_t,t\geq0\}$ and $\eta_2$ exists by the strong Markov property, and so on. From this, we know that $\eta_n$ is well defined. Let \[\eta:=\lim_{n\rightarrow +\infty} \eta_n\] be {\bf the first flying time} of $\{X_t,t\geq0\}$. We assume throughout the paper that the first flying time is infinite, so that the process is uniquely determined by the Q-matrix.

The {\bf embedded Markov chain} of CTMC is an important concept. Given a CTMC, its embedded Markov chain $Y_n$ is a discrete-time Markov chain (DTMC), which can be obtained by considering the process at the jump times. Intuitively, $Y_n = X_{T_n}$, where $T_n$ are the successive jump times of the process $X_t$. The transition probabilities of the embedded Markov chain are related to Q-matrix $Q=(q_{xy})$ of CTMC. The one-step transition probability $p_{xy}$ of the embedded Markov chain from state $x$ to state $y$ is given by $p_{xy}=q_{xy}\slash q_{x}$ for $x\neq y$, and $p_{xx} = 0$. The embedded chain has attracted extensive attention due to its wide applications in various practical models\cite{liu2015embedded,athreya1968embedding} and approximation theory\cite{bottcher2014embedded}.

In the theoretical framework of CTMCs, the invariant measure is one of the core concepts for characterizing the long-term dynamic behavior of a system. For a CTMC with a countable state space $I$, if there exists a non-negative measure \(m(\{i\})_{i \in I} \) such that for any state \( j \in I \) and any \( t \geq 0 \), the equation  
\[m(\{j\}) = \sum_{i \in I} m(\{i\})P(t,i,j)\]  
holds, then  \(m(\{i\})_{i \in I} \) is called the {\bf invariant measure}. When the state space is finite, the existence of an invariant measure is guaranteed by the basic theory of linear algebra. When the state space is infinite, the existence of an invariant measure is closely related to the recurrence property of the chain. 

Here we present some classical characterizations of recurrence and null recurrence for irreducible CTMCs. For convenience, we denote $P(\cdot|X_0=x)$ and $E(\cdot|X_0=x)$ by $P^x(\cdot)$ and $E^x(\cdot)$.
\begin{thm} \label{thm:2.1}
For an irreducible CTMC $\{X_t,t\geq 0\}$ with state space $I$, the following statements are equivalent.
    \begin{enumerate}[(1)]
        \item $X_t$ is recurrent.
        \item For any $x,y\in I$, $\int_0^\infty P_{xy}(t) dt=\infty.$ 
        \item For any $x,y\in I$, $P^x(T_y<\infty)=1.$
        \item For any $x\in I$, $P^x(\tau_x<\infty)=1.$
        \item The embedded chain $Y_n$ is recurrent. 
    \end{enumerate}
\end{thm}

For recurrent CTMCs, according to the dynamic behavior of the process, it can be further divided into null recurrence and positive recurrence. There is already a rich literature on the characterization of these two properties. 
\begin{thm} \label{thm:2.2}
For an irreducible recurrent CTMC $\{X_t,t\geq 0\}$ with state space $I$, the following statements are equivalent.
    \begin{enumerate}[(1)]
        \item $X_t$ is null-recurrent.
        \item For any $x\in I$, $\lim_{t\rightarrow\infty}P_{xx}(t)=0$. 
        \item For any $x\in I$, $E^x(\tau_x)=\infty$.
        \item The invariant measure $m(\{x\})$ satisfies $m(I)=\infty$.
    \end{enumerate}
\end{thm}

In fact, there are many characterizations from other perspectives. We only present those that are needed in this paper. For detailed proofs, refer to \cite{anderson2012continuous,chung1967markov}. 

In the subsequent discussion in this paper, we impose a special constraint condition on \(w(\cdot)\) as follows, which is denoted by Condition (C).  

Condition (C): \(w(t)\) is a non-negative increasing function on \([0,+\infty)\) and satisfies 
\begin{align}
   \lim_{t\rightarrow \infty} w'(t)=0,\,\,w''(t)<0 \mbox{ on }[t_0,+\infty) \mbox{ for some }t_0>0. \tag{C}
\end{align} 
For example, $w(t)=1\vee\log^a t(a>0)$, $w(t)=t^a(0<a<1)$, $\cdots$
Intuitively speaking, a function that satisfies Condition (C) is an increasing function whose growth rate is slower than that of a linear function. Inspired by \cite{Zhao2006null}, the first returning time can be characterized more precisely by considering $E^x(w(\tau_x))$ for different weight functions $w(\cdot)$ satisfying Condition (C).

\section{The first returning speed of CTMCs}

In this section, we will make an estimation for $E^x(w(\tau_x))$. In this way, some equivalent characterizations of $E^x(w(\tau_x))<\infty$ will be obtained. 

On the probability triple $(\Omega,\ff,P)$, let $\{X_t,t\geq 0\}$ be a CTMC with a countable state space $I$. Let $\{G_{\alpha}(x,y),\alpha\geq 0\}$ be the resolvent of $P(t,x,y)$. Particularly for $x,y\in I$, 
\[G_{\alpha}(x,y):=\int_0^\infty e^{-\alpha t}P(t,x,y) dt.\]
For $x,y\in I$, we define a function on $I\times I$ as  
\[S_t (x,y):=\int_0^t P(s,x,y)ds.\] 
Obviously, we have \[\lim_{t\rightarrow +\infty} S_t(x,y)=\lim_{\alpha\downarrow 0}G_{\alpha}(x,y):=G(x,y).\]
Let $F_{xx}(t):=P^x(\tau_x\leq t)$ and $F_{yx}(t):=P^y(T_x\leq t)$ for $x\neq y$, which are 
the distribution functions of $\tau_x$ and $T_x$. According to the Markov property, we have 
\begin{align}\label{equ:3.1}
    P(t,y,x)=\delta_{yx} e^{-q_xt}+\int_0^t P(t-s,x,x)dF_{yx}(s).
\end{align}
In fact, the two random variables $T_x$ and $\tau_x$ are the convolutions of a series of exponential distributions and transition matrices, so both of them have probability density functions. Taking Laplace transforms in (\ref{equ:3.1}), we have  
\[G_\lambda(x,x)=\dfrac{1}{\lambda+q_x}+G_\lambda(x,x)\int_0^{\infty}e^{-\lambda s}dF_{xx}(s),\]
and rearranging gives us 
\begin{align}
    G_\lambda(x,x)&=\dfrac{1}{\lambda+q_x}\dfrac{1}{1-\int_0^{\infty}e^{-\lambda t}dF_{xx}(t)}  \label{equ:3.2}\\
    &= \dfrac{1}{\lambda(\lambda+q_x)}\dfrac{1}{\int_0^{\infty}e^{-\lambda t}(1-F_{xx}(t))dt}.\label{equ:3.3}
\end{align}
Denote the integral of the tail probability distribution of $\tau_x$ by 
\[B_t(x,x):=\int_0^t (1-F_{xx}(s)) ds.\] 
We will utilize these relationships to estimate $E^x(w(\tau_x))$ from different perspectives. Before making subsequent estimates, we need an analytical lemma. 

\begin{lem}\label{lem:3.1}
     Fixed $T>0$. If a non-negative function $f(t)$ is integrable on $[0,+\infty)$ and satisfies that 
    \[\forall a>0,\,\int_a^{a+T} f(t) dt\leq \int_0^{T} f(t) dt, \]  
    then 
    \[1-\dfrac{1}{e}\leq \dfrac{\int_0^T f(t)dt}{\int^\infty_0 e^{-\frac{t}{T}}f(t)dt} \leq e.\]
\end{lem}
\begin{proof}
    
For the upper bound, we have 
\begin{align*}
    \int^\infty_0 e^{-\frac{t}{T}}f(t)dt \geq  \int^T_0 e^{-\frac{t}{T}}f(t)dt
     \geq e^{-1}\int^T_0 f(t)dt.     
\end{align*}
For the lower bound, we have 
\begin{align*}
    \int^\infty_0 e^{-\frac{t}{T}}f(t)dt &= \sum_{M=0}^\infty \int^{(M+1)T}_{MT} e^{-\frac{t}{T}}f(t)dt \\
    &\leq \sum_{M=0}^\infty e^{-M}\int^{(M+1)T}_{MT} f(t)dt \\
    &\leq \sum_{M=0}^\infty e^{-M}\int^{T}_{0} f(t)dt\\
    &= \dfrac{1}{1-e^{-1}} \int^{T}_{0} f(t)dt.
\end{align*}
\end{proof}

 By Lemma \ref{lem:3.1}, we can give the following propositions. 

\begin{prop}\label{prop:3.2}
For $t>0$ sufficiently large, we have following estimations.  
\begin{enumerate}[(1)]
    \item  $1-\dfrac{1}{e}\leq \dfrac{S_t(x,x)}{G_{1/t}(x,x)} \leq e.$
    \item  $1-\dfrac{1}{e}\leq \dfrac{B_t(x,x)}{\int^\infty_0 e^{-\frac{s}{t}}(1-F_{xx}(s))ds} \leq e.$
    \item  $\dfrac{1}{2q_x}\leq \dfrac{G_{1/t}(x,x) \int^\infty_0 e^{-\frac{s}{t}}(1-F_{xx}(s))ds }{t} \leq \dfrac{1}{q_x}.$
\end{enumerate}
\end{prop}
\begin{proof}
    \begin{enumerate}[(1)]
        \item For any $x,y\in I,T>0,$     
    \begin{align*}
        \int_0^T P(t,y,x)dt&=\int_0^T\int_0^t P(t-s,x,x)dF_{yx}(s)dt\\
        &=\int_0^T dF_{yx}(s)\int_s^T P(t-s,x,x)dt\\
        &\leq \int_0^T P(t,x,x)dt.
    \end{align*}
    Apply the above inequality, for any $a,T>0$, we have
    \begin{align*}
        \int_{a}^{a+T} P(t,x,x)dt&=\int_{a}^{a+T} \sum_{y\in I} P(a,x,y)P(t-a,y,x)dt\\
        &= \sum_{y\in I} P(a,x,y)\int_{a}^{a+T} P(t-a,y,x) dt\\
        &= \sum_{y\in I} P(a,x,y)\int_{0}^{T} P(t,y,x) dt\\
        &\leq \sum_{y\in I} P(a,x,y)\int_{0}^{T} P(t,x,x) dt\\
        &=\int_{0}^{T} P(t,x,x)dt.
    \end{align*}
    Thus, by Lemma \ref{lem:3.1}, the proof is completed.

    \item Since $B_t(x,x)=\int_0^t P^x(\tau_x>s) ds$, for any $a,T>0$, it satisfies 
     \[\int_a^{a+T} P^x(\tau_x>s) ds\leq \int_0^{T} P^x(\tau_x>s) ds. \]  
     By Lemma \ref{lem:3.1}, we obtain the second proposition.

    \item By (\ref{equ:3.3}), take $\lambda=1/t$, then we have 
    \[\dfrac{G_{1/t}(x,x) \int^\infty_0 e^{-\frac{s}{t}}(1-F_{xx}(s))ds }{t}=\dfrac{1}{\dfrac{1}{t}+q_x}.\] 
    For $t>0$ sufficiently large, the third proposition is obviously valid.
    \end{enumerate}
\end{proof}

For two non-negative measurable functions $f(t)$ and $g(t)$ on $[0,+\infty)$, we say that \(f(t)\sim g(t)\) if there exist constants $C_2>C_1>0$ such that \(C_1g(t)\leq f(t) \leq C_2g(t)\) for all sufficiently large $t$. Furthermore, if \(f(t)\sim g(t)\), then the convergence properties of \(\int^\infty f(t) dt\) and \(\int^\infty g(t) dt\) are the same. For such two integrals, we also write \(\int^\infty f(t) dt\sim \int^\infty g(t) dt.\)

The most important result in this section will be presented. We will utilize the propositions above to give an estimate of $E^x(w(\tau_x))$.  

\begin{thm}\label{thm:3.3}
    If $w(\cdot)$ satisfies (C), then the following conditions are mutually equivalent. 
\begin{enumerate}[(1)]
    \item $E^x(w(\tau_x))<\infty.$
    \item $\int^\infty |w''(t)|\frac{t}{S_t(x,x)} dt<\infty.$
    \item $\int^\infty |w''(t)|\frac{t}{G_{1/t}(x,x)} dt<\infty.$
    \item $\int^\infty |w''(t)| (t-tE^x(e^{-\tau_x/t}))dt<\infty.$
\end{enumerate}
\end{thm}

\begin{proof}
    By Fubini Theorem, for any $k>0$, we have
    \begin{align*}
        E^x(w(\tau_x)-w(k);\tau_x\geq k)&=E^x\left(\int_k^\infty w'(t)1_{\{t\leq \tau_x\}} dt\right)\\
        &= \int_k^\infty w'(t)P^x(\tau_x\geq t) dt\\
        &= \int_k^\infty -w''(t) dt\int_k^t P^x(\tau_x\geq s) ds.
    \end{align*}
Due to Proposition \ref{prop:3.2}, we have 
\begin{align*}
    \int_k^t P^x(\tau_x\geq s)ds &\sim \int_0^t (1-F_{xx}(s)) ds \\
    &\sim\dfrac{t}{S_t(x,x)}\sim \dfrac{t}{G_{1/t}(x,x)}.
\end{align*}
By (\ref{equ:3.2}), we have 
\[\dfrac{t}{G_{1/t}(x,x)}=(1+tq_x)(1-E^x(e^{-\tau_x/t}))\sim t-tE^x(e^{-\tau_x/t}).\]
That completes the proof.
\end{proof}
Due to the Ratio Limit Theorem, which we can refer to section II.12 in \cite{chung1967markov}, if the above formula holds for a certain \(x\in I\), then it holds for all \(x\in I\). Therefore, $E^x(w(\tau_x))<\infty$ is a class property. In particular, if $w(t)=t^a$ and $0<a<1$, we have 
\begin{align}
    E^x(\tau_x^a)<\infty&\Leftrightarrow \int^\infty \frac{t^{a-1}}{S_t(x,x)} dt<\infty \label{equ:3.4}\\ & \Leftrightarrow \int^\infty \frac{t^{a-1}}{G_{1/t}(x,x)} dt<\infty \label{equ:3.5}
    \\ &\Leftrightarrow \int^\infty t^{a-1}(1-E^x(e^{-\tau_x/t}))dt<\infty. \label{equ:3.6}
\end{align}

\begin{ex}
    Consider a birth-death process with the state space $I=\zz$. Each state $n$ transitions to $n+1$ and $n-1$ at a rate of $1$, that is, $q_{n,n+1}=q_{n,n-1}=1$. The invariant measure of the process is not a probability measure, so the process is null recurrent. The embedded chain of this process is a simple random walk. 

    For the function $f:\zz \rightarrow \re$, we have 
    \[Qf(n)=f(n+1)+f(n-1)-2f(n).\] Perform the Fourier transform on the function $f(n)$, where $F(\theta)=\sum_{n\in \zz} f(n)e^{in\theta}$, then we have 
    \begin{align*}
    \ff(Qf)(\theta)&=\sum_{n\in\zz} f(n)e^{in\theta}(e^{i\theta}+e^{-i\theta}-2)\\
    &=(2\cos\theta-2)F(\theta),
    \end{align*}
    so that $\ff(G_{\lambda})=(\lambda-(2\cos\theta-2))^{-1}$.
    By the inverse Fourier transform, we can obtain 
    \[G_{\lambda}(n,m)=\dfrac{1}{2\pi}\int_{-\pi}^\pi \dfrac{e^{i(n-m)\theta}}{\lambda+4\sin^2(\theta/2)} d\theta=\dfrac{1}{\sqrt{\lambda(\lambda+4)}}\left(\dfrac{\lambda+2-\sqrt{\lambda(\lambda+4)}}{2}\right)^{|n-m|}.\]
    By Theorem \ref{thm:3.3}, 
    \begin{align*}
        \int_k^\infty \frac{t^{a-1}}{G_{1/t}(x,x)} dt= \int_k^\infty t^{a-1}\sqrt{\dfrac{1}{t}\left(\dfrac{1}{t}+4
        \right)}dt \sim \int^\infty t^{a-1}\dfrac{2}{\sqrt{t}} dt.  
    \end{align*}
     Hence, for $a>0$, we have $E^x(\tau_x^a)<\infty$ if and only if $a<1/2$.
\end{ex}

\begin{ex}
 Consider a CTMC with the state space $I=\zz^2$. Each state $(m,n)$ transitions to the four adjacent lattice points $(m,n\pm1)$ and $(m\pm 1,n)$ at a rate of $1$. it is easy to check the process is null recurrent. The embedded chain of this process is two-dimensional simple random walk.

 Similarly, through the Fourier transform, we can obtain 
 \[G_{\lambda}((0,0),(k,l))=\dfrac{1}{4\pi^2} \int_{-\pi}^{\pi}\int_{-\pi}^{\pi}\dfrac{e^{i(k\theta_1+l\theta_2)}}{\lambda+1-\frac{1}{2}(\cos \theta_1+\cos \theta_2)}d\theta_1d\theta_2.\]
 According to Chapter 15 in \cite{spitzer2001principles}, we have 
 \[G_{1/t}((0,0),(0,0))\sim \dfrac{1}{\pi}\log 4t.\]
 If we take $w(t)=t^a(0<a<1)$, then the integral in (\ref{equ:3.5}) is always divergent. Assume that $w(t)=\log^at$ where $a>0$. By Theorem \ref{thm:3.3}, consider that 
 \begin{align*}
     \int^\infty |w''(t)|\frac{t}{G_{1/t}((0,0),(0,0))} dt&\sim  \int^\infty \dfrac{-a(a-1-\log t)(\log t)^{a-2}}{t^2}\dfrac{\pi t}{\log t} dt\\
     &=a(1-a)\pi \int^\infty \dfrac{(\log t)^{a-3}}{t} dt+ a\pi\int^\infty \dfrac{(\log t)^{a-2}}{t} dt\\
     &\sim \int^\infty \dfrac{(\log t)^{a-2}}{t} dt.
 \end{align*}
Hence, for $a>0$, we have $E^x(\log^a\tau_x)<\infty$ if and only if $a<1$.
\end{ex}

\section{Reversible CTMCs and electric networks}
In this section, we establish the connection between reversible CTMCs and electric networks. By utilizing the properties of electric networks, we will provide a characterization of null-recurrence and first returning speed. Meanwhile, some general estimates of effective conductance for a class of electric networks will be obtained.

\begin{defi}
    Let $(I,E)$ be a connected graph, endowed with non-negative edge weights $C=\{c_e:e\in E\}$, which is called conductance. A triple $(I,E,C)$ is called \textbf{an electric network}.
\end{defi}
 For discrete-time Markov chains, when it comes to characterizing recurrence and transience using electric networks, one can refer to \cite{nachmias2020planar}\cite{doyle1984random}\cite{Zhao2006null}. We expound the relationship between CTMC and the electric network in a closer way. 

On the probability triple $(\Omega,\ff,P)$, let $\{X_t,t\geq 0\}$ be a reversible and irreducible CTMC with a countable state space $I$ provided that there exists a measure $\{m_x,x\in I\}$ such that \[m_xP_{xy}(t)=m_yP_{yx}(t), \,\,x,y\in I,\,t\geq 0. \]
This amounts to 
\[m_xq_{xy}=m_yq_{yx},\,\,x,y\in I. \]
We start with a given $Q-$matrix $(q_{xy})$ to construct the corresponding electric network.

For the electric network $(I,E,C)$, we make the vertex set $I$ coincide with the state space of $\{X_t,t\geq 0\}$, and the edge set is composed of those pairs of vertices that are connected with a positive rate
\[E=\{xy:x,y\in I, q_{xy}>0\}.\] We define the conductance as 
\[c_{xy}=m_xq_{xy},\] and resistance $r_{xy}=c_{xy}^{-1}$. We use the notation \(c_x=\sum_{y\in I} c_{xy}.\) In this way, we define an electric network corresponding to the CTMC. 
\begin{defi}
    Given a network $(I,E,C)$ and two boundaries $a,b\in I$, \textbf{a voltage} $v(x)$ is a function $v:I\rightarrow \re$ that is harmonic at any $x\in I\backslash\{a,b\}$. We say $v(x)$ as the \textbf{unit voltage} between $a$ and $b$, that is, $v$ is the voltage, and $v(a)=1,v(b)=0.$
\end{defi}
Fixing the values at the two boundaries, the voltage $v$ is unique by the property of harmonic function. 

Denote $R_{x\leftrightarrow y}$ by the {\bf effective resistance} between $x$ and $y$ in the network. Effective resistance captures the equivalent resistance between two nodes in a network, considering all possible paths (series/parallel combinations). It links to recurrence and transience in Markov chains. Its reciprocal is the {\bf effective conductance}, denoted by $C_{x\leftrightarrow y}$. Simply say, recurrence and transience are equivalent to whether the effective conductance of the electric network is finite. 

For more refined characterizations such as null recurrence, we refer to the method in \cite{Zhao2006null}, we construct the electric networks to calculate $E^x(e^{-\lambda T_a})$. Taking a new state $\Delta\notin I$, define \(I_{\Delta}=I\cup\{\Delta\}\) and \(E_{\Delta}=E\cup\{x\Delta,\Delta x; x\in I\}\). Regard \(\{X_t, t\geq0\}\) as a CTMC on \(I_{\Delta}\), where \(\Delta\) is an absorbing state, that is, \(P_{\Delta}(X_t=\Delta) = 1\) for $t\geq 0$. For a parameter $\lambda>0$, we define the conductance as 
\[c_{x\Delta}=c_{\Delta x}=\dfrac{\lambda}{q_x}c_x.\]Intuitively speaking, all the nodes in $I$ undergo a grounding process. Denote the corresponding new network by $(I_{\Delta},E_{\Delta},C_{\Delta})$. Different from the original process \(\{X_t, t\geq0\}\), we use $P_{\Delta}$ to denote the transition probability on the new network, then we have 
\[P_{\Delta}(x,y)=\begin{cases}
    \dfrac{q_x}{q_x+\lambda}\dfrac{c_{xy}}{c_x}, &x,y\neq \Delta,\\
    \dfrac{\lambda}{q_x+\lambda},&x\neq\Delta, y=\Delta.
\end{cases}\]
According to Theorem 3.2.1 in \cite{Zhou2020hitting}, we have 
\begin{align}\label{equ:4.1}
    E^x(e^{-\lambda \tau_x})=1-\dfrac{q_x}{q_x+\lambda}\dfrac{C_{x\leftrightarrow\Delta}(\lambda)}{c_x}.
\end{align}
Here we use the notation $C_{x\leftrightarrow\Delta}(\lambda)$ to denote the effective conductance because we will take $\lambda$ as a variable subsequently. 
\begin{defi}
    The \textbf{energy} of a function $f:I\rightarrow \re$ in $(I_{\Delta},E_{\Delta},C_{\Delta})$, denoted by $\ene(f)$, is defined as \[\ene(f):=
    \dfrac{1}{2}\sum_{x,y\in I} (f(x)-f(y))^2c_{xy}+\sum_{x\in I} f^2(x)\dfrac{\lambda}{q_x}c_{x}.\]
\end{defi}
Following Dirichlet's principle of minimum energy and Theorem 3.2.2 in \cite{Zhou2020hitting}, we have 
\begin{align}\label{equ:4.2}
    C_{a\leftrightarrow \Delta}(\lambda)=\inf\{\ene(f):f(a)=1\}.
\end{align}
The energy attains the minimum value if and only if \(f\) is the unit voltage from \(a\) to \(\Delta\) in \((I_{\Delta}, E_{\Delta}, C_{\Delta})\), which is $f(x)=E^x(e^{-\lambda T_a})$. Thus, we are able to characterize $E^x(w(\tau_x))<\infty$ from the perspective of electric networks, which is the first main result in this section.

\begin{thm}\label{thm:4.4}
    For a CTMC \(\{X_t, t\geq0\}\) on $I$, define the electric network \((I_{\Delta}, E_{\Delta}, C_{\Delta})\) as above. If $w(x)$ satisfies Condition (C), then the following two assertions are equivalent. 
    \begin{enumerate}[(1)]
        \item For $x\in I$, $E^x(w(\tau_x))<\infty$.
        \item For $x\in I$, $\int^\infty t |w''(t)|C_{x\leftrightarrow\Delta}(1/t) dt<\infty.$
    \end{enumerate}
\end{thm}
\begin{proof}
    By (\ref{equ:4.1}), we have \[E^x(e^{-\tau_x/t})=1-\dfrac{tq_x}{tq_x+1}\dfrac{C_{x\leftrightarrow\Delta}(1/t)}{c_x}. \]
    By substituting the fourth condition from Theorem \ref{thm:3.3}, we can obtain 
    \begin{align*}
        \int^\infty |w''(t)| (t-tE^x[(\exp(-1/t))^{\tau_x}])dt&=\int^\infty |w''(t)|\dfrac{t^2q_x}{tq_x+1}\dfrac{C_{x\leftrightarrow\Delta}(1/t)}{c_x}dt\\
        &\sim \int^\infty  t |w''(t)|C_{x\leftrightarrow\Delta}(1/t) dt.
    \end{align*}
\end{proof}
If we can obtain the asymptotic estimate of \(C_{x\leftrightarrow\Delta}(1/t)\) through the techniques of the electric networks, then we can get the characterization of \(E^x(w(\tau_x))<\infty\). In Section 5, we will make use of this idea to deal with the birth-death process. 

We now state two important techniques to estimate the effective conductance.
\begin{defi}[{\bf Short circuit}]
    There exists an equivalence relation $\sim$ on \(I\). Let \(\widetilde{x}\) be the equivalence class of \(x\), that is, \(\widetilde{x}=\{y\in I:y\sim x\}\). Let \(I^*=\{\widetilde{x}:x\in I\}\), \(c_{\widetilde{x}\widetilde{y}}^*=\sum_{x\in\widetilde{x},y\in\widetilde{y}}c_{xy}\), and \(E^* = \{\widetilde{x}\widetilde{y}:c_{\widetilde{x}\widetilde{y}}^*>0\}\). Then, we say that the electric network \((I^*, E^*, C^*)\) is obtained by short-circuiting the electric network \((I, E, C)\).  
\end{defi}
\begin{defi}[{\bf Open circuit}]
If \(I^* \subset I\), \(E^* \subset E\), and for all \(xy \in E^*\), \(C_{xy}^* = C_{xy}\), then we say the electric network \((I^*, E^*, C^*)\) is obtained by open-circuiting the electric network \((I, E, C)\). 
\end{defi}

We can regard a short circuit as the conductance between two nodes in the original network becoming $\infty$, and an open circuit as the conductance between two nodes in the original network becoming $0$. According to 2.3 in \cite{lyons2017probability}, we have Rayleigh’s Monotonicity Principle for $(I_{\Delta},E_{\Delta},C_{\Delta})$. 

\begin{thm}\label{thm:4.7}
If \((I_{\Delta}, E_{\Delta}, C_{\Delta})\) is an electric network, then we have the following statement. 
    \begin{enumerate}
        \item Short-Circuit Rule: If \((I^*_{\Delta}, E^*_{\Delta}, C^*_{\Delta})\) is obtained by short-circuiting \((I_{\Delta}, E_{\Delta}, C_{\Delta})\), then for any \(x \in I\), the effective conductance from \(x\) to \(\Delta\) satisfies  
    \[C_{x \to \Delta}^*(\lambda) \geq C_{x \to \Delta}(\lambda).\]  
    (Intuitively, short-circuiting nodes increases the effective conductance.)
    \item Open-Circuit Rule: If \((I^*_{\Delta}, E^*_{\Delta}, C^*_{\Delta})\) is obtained by opening \((I_{\Delta}, E_{\Delta}, C_{\Delta})\), then for any \(x \in I\), the effective conductance from \(x\) to \(\Delta\) satisfies  \[C_{x \to \Delta}^*(\lambda) \leq C_{x \to \Delta}(\lambda).\]  
    (Intuitively, opening edges or removing nodes decreases the effective conductance.)
    \end{enumerate}
\end{thm}

Define a notation for non-negative numbers. For any finite sequence \(a_1, a_2, \ldots, a_n \geq 0\), let $[a_1, a_2, \cdots, a_n] = (\sum_{i=1}^n \frac{1}{a_i})^{-1}.$  
For an infinite non-negative sequence \(\{a_i\}\), define  
$[a_1, a_2, \dots] = (\sum_{i=1}^\infty \frac{1}{a_i})^{-1},$ which directly represents the total conductance of \(\{c_i\}\) connected in series.

For a half-line electric network defined on a countable number of vertex, which can be regarded as being on \(\mathbf{N}\), we can provide an estimation of the effective conductance $C_{0\leftrightarrow \Delta}$ in \((I_{\Delta}, E_{\Delta}, C_{\Delta})\). The corresponding electric network \((I_{\Delta}, E_{\Delta}, C_{\Delta})\) can be considered as shown in the following figure. We denote $c_{i,i+1}$ by $a_i$ and denote $c_{i\Delta}$ by $b_i$ for $i\ge 0$. Let $A_n=[a_0,a_1,\cdots,a_n]$ and $B_n=\sum_{i=0}^n b_i$. 
\begin{figure}[htbp]
    \centering
    \includegraphics[width=0.4\linewidth]{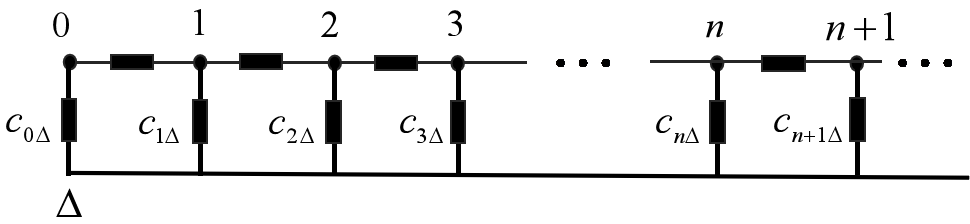}
    \caption{Electric network on $\nat$}
    \label{fig:1}
\end{figure}

\begin{thm} \label{thm:4.8}
For any $n\geq 0$,
$$ [A_n,B_n]\leq  [A_{n-1},B_n] \leq C_{0\leftrightarrow \Delta}\le A_n+B_n\leq A_{n-1}+B_{n}.$$
\end{thm}
\begin{proof}
    Firstly, short the vertex $\{\Delta,n+1,\cdots\}$ together by $\Delta$, and use $C_1$ to denote the effective conductance
    between $0$ and $\Delta$ in this new network as shown in the following figure.
    \begin{figure}[htbp]
    \centering
    \includegraphics[width=0.4\linewidth]{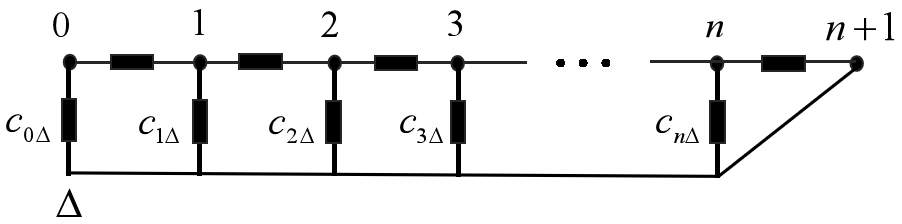}
    \caption{Short-circuiting electric network on $\nat$}
    \label{fig:2}
    \end{figure}
    
    By Theorem \ref{thm:4.7}, we have $C_1\ge C_{0\leftrightarrow \Delta}$.
    In this new network, put a unit voltage $u_i$, as a brief notation for $u(i)$,  from $0$ to $\Delta$ such that $u_0=1$ and $u_\Delta=0$. Then the voltage at $x$ is $u_x=P^x(T_0<T_\Delta)$. Clearly, $1=u_0>u_1>\cdots >u_{n+1}=\cdots=u_\Delta=0$.
    Thus by (\ref{equ:4.2}), we have 
    \begin{align*}
    C_1&=\inf\left\{ \sum_{i=0}^{n} \left[a_i (u_i-u_{i+1})^2+  b_i u^2_i\right]: 1=u_0>\cdots >u_{n+1}=\cdots=u_\Delta=0\right\}\\
    &\le B_n+ \inf\left\{\sum_{i=0}^{n} a_i (u_i-u_{i+1})^2:1=u_0>\cdots >u_{n+1}=0\right \}\\
    &=B_n+\inf\bigg\{\sum_{i=0}^{n} a_i (u_i-u_{i+1})^2:u_0=1,u_{n+1}=0\bigg\}\\
    &=A_n+B_n.
    \end{align*}
    According to the definition of \( A_n \), $A_n+B_n\leq A_{n-1}+B_{n}$ holds naturally.  
    Intuitively speaking, the validity of this inequality lies in comparing the effective conductance in Figure \ref{fig:2} with that in Figure \ref{fig:3}. This can be seen from the inequality in the second step.
    \begin{figure}[htbp]
    \centering
    \includegraphics[width=0.5\linewidth]{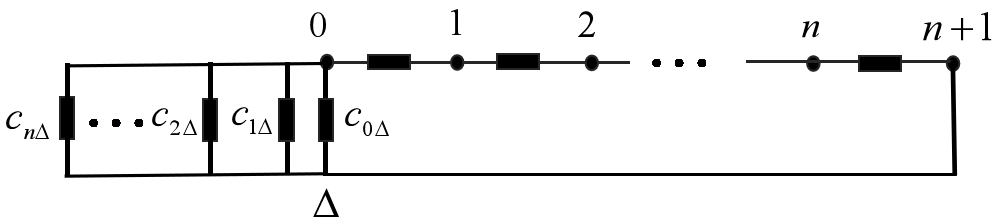}
    \caption{Electric network for upper bound}
    \label{fig:3}
    \end{figure}

    Secondly, cut the edges $\{(i,i-1),(i,\Delta):i\ge n+1\}$, and use $C_2$ to denote the effective conductance between $0$ and $\Delta$ in this new network as shown in Figure \ref{fig:4}. 
    \begin{figure}[htbp]
    \centering
    \includegraphics[width=0.38\linewidth]{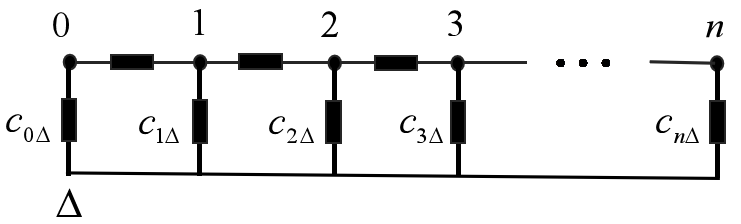}
    \caption{Open-circuiting electric network on $\nat$}
    \label{fig:4}
    \end{figure}
   
    By Theorem \ref{thm:4.7}, we have $C_2\le C_{0\leftrightarrow \Delta}$.
    Similarly,
    \begin{align*}
    C_2&=\inf\left\{ \sum_{i=0}^{n-1} \left[a_i (u_i-u_{i+1})^2+  b_i u^2_i\right]+b_n u^2_n: 1=u_0>\cdots >u_n>u_\Delta=0\right\}\\
    &\ge \inf\left\{ B_nu^2_n+ \sum_{i=0}^{n-1} a_i (u_i-u_{i+1})^2:1=u_0>\cdots >u_n\right \}\\
    &= \inf\left\{ B_n(u_n-u_\Delta)^2+ \sum_{i=0}^{n-1} a_i (u_i-u_{i+1})^2:u_0=1,u_\Delta=0 \right\}\\
    &=[A_{n-1},B_n]\geq [A_n,B_n].\\
    \end{align*}
    Intuitively speaking, this inequality lies in comparing the effective conductance in Figure \ref{fig:4} with that in Figure \ref{fig:5}. 
    \begin{figure}[htbp]
    \centering
    \includegraphics[width=0.4\linewidth]{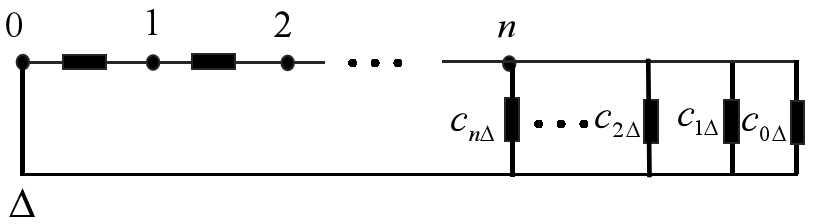}
    \caption{Electric network for lower bound}
    \label{fig:5}
    \end{figure}

    That completes the proof.
\end{proof}

Using similar techniques, for an electric network on a straight line, which can be regarded as being on \(\mathbf{Z}\), we can also provide an estimation of the effective conductance $C_{0\leftrightarrow \Delta}$ in \((I_{\Delta}, E_{\Delta}, C_{\Delta})\). The corresponding electric network \((I_{\Delta}, E_{\Delta}, C_{\Delta})\) can be considered as shown in the following figure. Denote $c_{i,i+1}$ by $a_i$ and denote $c_{i\Delta}$ by $b_i$ for $i\in \zz$. Let $A_n^+=[a_0,a_1,\cdots,a_n],A_n^-=[a_{-1},a_{-2},\cdots,a_{-n}]$ and $B_n^+=\sum_{i=0}^n b_i, B_n^-=\sum_{i=1}^n b_{-i}$. 
\begin{figure}[htbp]
    \centering
    \includegraphics[width=0.5\linewidth]{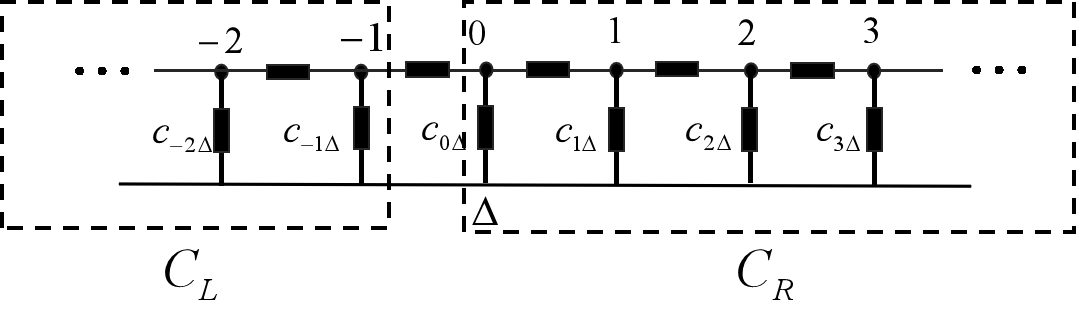}
    \caption{Electric network on $\zz$}
    \label{fig:6}
\end{figure}

As shown in the figure above, we denote the effective conductance of the electric network corresponding to the natural number part as $C_R$, and the effective conductance corresponding to the electric network corresponding to the negative integer part as $C_L$.  

\begin{thm} \label{thm:4.9}
For any $n,m\geq 0$, 
$$ [A_n^+,B_n^+]+[A_m^-,B_m^-]\leq C_{0\leftrightarrow \Delta}\le A_n^++A_m^-+B_n^++B_m^-.$$
\end{thm}

\begin{proof}
We regard the electric network on $\zz$ as the parallel connection of the left and right two electric networks. By Theorem 4.8, we can obtain \[[A_n^+,B_n^+]+[A_m^-,B_m^-]\le C_{0\leftrightarrow \Delta}=C_R+[a_{-1},C_L]\le A_n^++A_m^-+B_n^++B_m^-.\] 

\end{proof}
These two important estimations will help us derive many significant results on birth-death process in Section 5. 

\section{The first returning speed of birth-death process}
    The birth-death process, as a fundamental model in many fields, is a very important class of CTMC. The electric network corresponding to birth-death process can be considered to be arranged on a straight line. In this section, we use the results given in Section 4 to present a general method for estimating the first returning speed of birth-death process. Based on this, we present the estimation results of some widely concerned processes.

    For a one-sided conservative birth-death process $\{X_t,t\geq 0\}$ on $\nat$, its density matrix $Q=(q_{ij})(i,j\in \nat)$ has the form
    \[Q=\begin{bmatrix}
        -\beta_0 & \beta_0 & 0 & \cdots &0 &0 &0 &\cdots  \\
        \alpha_1 & -(\alpha_1+\beta_1) &  \beta_1 & \cdots &0 &0 &0 &\cdots  \\
        \vdots & \vdots &\vdots &\, &\vdots &\vdots &\vdots &\,\\
        0 &0 &0 &\cdots &\alpha_n  & -(\alpha_n+\beta_n)& \beta_n &\cdots \\
        \vdots & \vdots &\vdots &\, &\vdots &\vdots &\vdots &\,\\
    \end{bmatrix},\]
    where $\beta_i>0(i\ge 0)$, $\alpha_i>0(i\ge 1)$. The invariant measure $m_i$ has the following representation
    \begin{align}\label{equ:5.1}
        m_0=1,\,m_n=\dfrac{\beta_0\beta_1\beta_2\cdots \beta_{n-1}}{\alpha_1\alpha_2\cdots \alpha_{n-1}\alpha_n}\mbox{\,\,for\,\,}n \ge 1,
    \end{align} 
    similar to the speed measure in one-dimensional diffusions. Let $M(n)=\sum_{i=0}^n m_i.$
    Let $s(n)$ denote the scale function defined by
    \begin{align}\label{equ:5.2}
    s(0)=0,\,s(1)=\beta_0^{-1},\,s(n)=\beta_0^{-1}+\sum_{i=1}^{n-1} \dfrac{\alpha_1\alpha_2\cdots\alpha_i}{\beta_0\beta_1\cdots\beta_i}\mbox{\,\,for\,\,}n \ge 2.
    \end{align}
    According to the definitions in Section 4, we can obtain 
    \begin{align}\label{equ:5.3}
        c_{n,n+1}=m_n\beta_n=\begin{cases}
        \beta_0, & n=0,\\
        \dfrac{\beta_0\beta_1\beta_2\cdots \beta_{n}}{\alpha_1\alpha_2\cdots \alpha_{n-1}\alpha_n}, & n \ge 1,\end{cases}
    \end{align}
    and $c_{n\Delta}=\lambda m_n.$
    
    For a null recurrent birth-death process, according to Section 6.8 in \cite{Wang1992Birth}, since the first flying time is equal to positive infinity with probability 1, the process is uniquely determined by the Q-matrix $Q=(q_{ij})$. The first flying time here is sometimes also referred to as the first infinity time or the first moment of bursting out. The first flying time is defined almost everywhere, and can be referred to Section 2.3 in \cite{Wang1992Birth}. The necessary and sufficient condition for null recurrence can also be seen in Section 6.8 in \cite{Wang1992Birth}, and it will not be elaborated here.  

    By Theorem \ref{thm:4.8}, the effective conductance $C_{0\leftrightarrow \Delta}$ can be estimated for a null-recurrent birth-death process. 
    We still denote $c_{i,i+1}$ by $a_i$ and denote $c_{i\Delta}$ by $b_i$ for $i\ge 0$. By the definition in Section 4, we have \[A_0=a_0=\beta_0,\,A_n=\dfrac{1}{s(n+1)},\] and \[B_0=b_0=\lambda m_0,\,B_n=\lambda M(n).\] Since what we need to estimate is the situation when \(\lambda\) is small, we can always assume $A_1> B_2$. For the electric network corresponding to a recurrent process, since the conductance $A_n$ should tend to 0, we have $\lim_{n\to\infty}A_n<\lim_{n\to\infty} B_n$. By Theorem \ref{thm:4.8}, it can be seen that the closer \(A_{n-1}\) is to \(B_n\), the better the estimation of \(C_{0\leftrightarrow \Delta}(\lambda)\) will be. Using this method, we present the estimation of some important models.
 
    In the subsequent proofs in this section, we briefly denote $E^0(\cdot)$ and $\tau_0$ by $E(\cdot)$ and $\tau$. The following theorem provides an estimation of first returning time for general null-recurrent birth-death processes. 

    \begin{thm}\label{thm:5.1}
        Assume that $\{X_t,t\geq 0\}$ is a null-recurrent birth-death process on $\nat$. Take $w(t)$ satisfied Condition (C). Let $F(n)=M(n)s(n)$ and \[F^{-1}(x):=\inf\{n\in\nat: F(n)\ge x\}.\]
        \begin{enumerate}[(1)]
        \item If \[\limsup_{n\rightarrow \infty} \dfrac{M(n)}{M(n-1)}<\infty,\]
        then $E(w(\tau))<\infty$ if and only if 
        \[\int^\infty |w''(t)|M(F^{-1}(t))dt<\infty.\]
            \item If \[\limsup_{n\rightarrow \infty} \dfrac{s(n)}{s(n-1)}<\infty,\]
          then $E(w(\tau))<\infty$ if and only if 
        \[\int^\infty \dfrac{t|w''(t)|}{s(F^{-1}(t))}dt<\infty.\]
        \end{enumerate}
    \end{thm}
    
    \begin{proof}
        Take $N=\inf\{n\geq 0:A_{n-1}\le B_n\}$. Since $A_1>B_2$, $A_n\downarrow 0$ and $B_n\uparrow \infty$, we have $3\le N<\infty$ and  
        \begin{align*}
            N &= \inf\{n\geq 0:\dfrac{1}{s(n)}\le \lambda M(n)\} \\
            &=\inf\{n\geq 0: M(n)s(n)\geq \dfrac{1}{\lambda}\}\\
            &=F^{-1}(\lambda^{-1}).
        \end{align*}
        By Theorem \ref{thm:4.8}, we have 
        \[\dfrac{A_{N-1}}{2}\leq [A_{N-1},B_N] \leq C_{0\leftrightarrow \Delta}(\lambda)\leq A_{N-1}+B_{N}\leq 2B_N,\]
        \[\dfrac{B_{N-1}}{2}\leq [A_{N-2},B_{N-1}] \leq C_{0\leftrightarrow \Delta}(\lambda)\leq A_{N-2}+B_{N-1}\leq 2A_{N-2}.\]
        It follows that 
        \begin{align}\label{eq5.4}
            \dfrac{B_{N-1}}{2}\le \dfrac{1}{2}\max(A_{N-1},B_{N-1}) \le C_{0\leftrightarrow \Delta}(\lambda)\leq 2\min\left( B_N,A_{N-2} \right)\leq 2B_N.
        \end{align}
        If \[\limsup_{n\rightarrow \infty} \dfrac{M(n)}{M(n-1)}<\infty,\]
        then there is a constant $k>0$ such that $\dfrac{M(n)}{M(n-1)}\le k$ for all $n\ge 1.$
        Thus, we have 
        \[ \dfrac{1}{2}\le \dfrac{C_{0\leftrightarrow \Delta}(\lambda)}{\lambda M(N-1)}=\dfrac{C_{0\leftrightarrow \Delta}(\lambda)}{B_{N-1}}\leq 2\dfrac{B_N}{B_{N-1}}=2\dfrac{M(N)}{M(N-1)}\le 2k.\]
        Therefore, take $\lambda=1/t$ and $N=F^{-1}(\lambda^{-1})$, by Theorem \ref{thm:4.4},
         \[E(w(\tau))<\infty\Leftrightarrow \int^\infty |w''(t)|M(F^{-1}(t))dt<\infty.\]
        The second statement in the theorem can also be deduced analogously from (\ref{eq5.4}).
    \end{proof}

    Here, we present an example of a null-recurrent birth-death process.
    \begin{ex}
    Consider the birth-death process $\{X_t,t\geq 0\}$ on $\nat$ with rates 
    \[\alpha_n=\beta_n=(n+1)^\gamma,\,\,\gamma \le 1.\]
    
    From (\ref{equ:5.1}) and (\ref{equ:5.2}), we have $s(n)=n$ and 
    \[m_0=1,\quad m_n=\dfrac{1}{(n+1)^{\gamma}},\,n\ge 1.\]
    By Theorem 3.5 in \cite{li2024ray}, $\gamma \le 1$ assures us that $\{X_t,t\geq 0\}$ is unique, honest and null-recurrent.  
    \begin{enumerate}[(1)]
        \item If $\gamma<1$, then $F(N)\sim N^{2-\gamma}$ and $F^{-1}(t)\sim t^\frac{1}{2-\gamma}.$ Let $w(t)=t^a(0<a<1)$, then by Theorem \ref{thm:5.1},  
        \[\int^\infty \dfrac{t|w''(t)|}{s(F^{-1}(t))}dt\sim \int^\infty t^{a-1+\frac{1}{\gamma-2}} dt.\]
        Hence $E(\tau^a)<\infty$ if and only if $a<\dfrac{1}{2-\gamma}$.
    
        \item If $\gamma=1$, then $F(N)\sim N\log N.$ By the property of Lambert W-function, $F^{-1}(t)\sim \dfrac{t}{\log t}.$ Let   $w(t)=t^a(0<a<1)$, then by Theorem \ref{thm:5.1}, 
        \[\int^\infty \dfrac{t|w''(t)|}{s(F^{-1}(t))}dt\sim \int^\infty \dfrac{\log t}{t^{2-a}} dt.\]
        Hence $E(\tau^a)<\infty$ always holds for $0<a<1$. 

        If we take $w(t)=\dfrac{t}{\log^b t}(b>0)$, then we have  \[\int^\infty \dfrac{t|w''(t)|}{s(F^{-1}(t))}dt\sim \int^\infty \dfrac{1}{t\log^b t } dt.\] Hence for $b>0$, $E(w(\tau))<\infty$ always holds for $b>1$. 
    \end{enumerate}
    \end{ex}

    \subsection{Linear growth model}
    In this section, we would discuss the critical case$: \gamma=1$, which is called linear growth model. The linear growth model\cite{karlin1958linear} is an important birth-death process and has spawned many other models\cite{di2016multispecies}\cite{kapodistria2016linear}. It is a model that has attracted widespread attention in biological and ecological systems and many other fields\cite{crawford2012transition}. In the literature, some results allow us to obtain the transient behaviors by using spectral method\cite{karlin1958linear} or the characteristic equation method\cite{zheng2004transient}. Using the route mentioned above, we conduct a more detailed characterization of the null recurrence of this process.

    Consider the birth-death process $\{X_t,t\geq 0\}$ on $\nat$ with rates 
    \[\alpha_n=kn+r, \beta_n=mn+s,\]
    where $s>0$, $r\ge 0$, $k>0$ and $m>0$ are constants denoting the immigration, emigration, death and birth rate per individual, respectively. We call such a birth-death process a {\bf $(k,r,m,s)$-linear growth model}. For the positive recurrence and recurrence of the process, we present a different approach.
    
    \begin{thm}
    Consider a $(k,r,m,s)$-linear growth model $\{X_t,t\ge 0\}$.
    \begin{enumerate}[(1)]
        \item The process is recurrent if and only if $k>m$ or $s-r\le k=m$. 
        \item The process is null recurrent if and only if $0\leq s-r\le k=m$.
        \item The process is transient if and only if $k<m$ or $s-r> k=m$.
    \end{enumerate}
    \end{thm}

    \begin{proof}
    By (\ref{equ:5.1}), we have   
    \[a_0=s, \,a_n=s\prod^{n}_{i=1}\dfrac{mi+s}{ki+r}\,(n\ge1).\]
    In the electric networks, the recurrence is equivalent to the resistance $R_{0\leftrightarrow \infty}=\sum _{i=0}^\infty a_i^{-1}$ being infinite.
    Since \[\lim_{n\rightarrow \infty} \dfrac{a^{-1}_{n+1}}{a^{-1}_n}=\lim_{n\rightarrow \infty}\dfrac{k(n+1)+r}{m(n+1)+s}=\dfrac{k}{m},\] according to the D'Alembert's Ratio Test, the series $\sum _{i=0}^\infty a_i^{-1}$ diverges when $k>m$, and converges when $k<m$.

    When $k=m$, we have 
    \[n\left(\dfrac{a_n^{-1}}{a_{n+1}^{-1}} -1\right)=\dfrac{(s-r)n}{mn+m+r}.\]
    According to the Raabe's Ratio Test\cite{huynh2021secondraabestestseries}, the series $\sum _{i=0}^\infty a_i^{-1}$ diverges when $s-r\le m$, and converges when $s-r>m$. That is to say, $\{X_t,t\ge 0\}$ is recurrent if and only if $k>m$ or $s-r\le k=m$.

    By (\ref{equ:5.1}), 
    \[\lim_{n\rightarrow \infty} \dfrac{m_{n+1}}{m_n}=\lim_{n\rightarrow \infty} \dfrac{mn+s}{k(n+1)+r}=\dfrac{m}{k}.\]
    Then for $m<k$, $\sum_{i=0}^\infty m_i$ converges. When $k=m$, we have 
    \[n\left(\dfrac{m_n}{m_{n+1}}-1\right)=\dfrac{(m+r-s)n}{mn+s}.\]
    In a similar way, by the Raabe's Ratio Test, $\sum_{i=0}^\infty m_i$ diverges if $0\leq s-r\le k=m$. 
    \end{proof}
    When \( r=0 \), the model degenerates into the one mentioned in Example 4.57 of \cite{chen2004markov}, and the result we obtain is a generalization. When the process is null recurrent, we provide a more delicate characterization.

    \begin{thm}\label{thm:5.4}
    Consider a null recurrent $(m,r,m,s)$-linear growth model $\{X_t,t\ge 0\}$ with $0\le s-r \le m$, 
    \begin{enumerate}[(1)]
        \item If $s-r=m$, $E(\log^\gamma\tau)<\infty$ if and only if $0<\gamma<1$.
        \item If $0\le s-r<m$, $E(\tau^a)<\infty$ if and only if $0<a<1-\dfrac{s-r}{m}$.
    \end{enumerate}
    \end{thm}

    \begin{proof}
    When $s-r=m$, by (\ref{equ:5.1}) and (\ref{equ:5.2}), we have 
    \begin{align*}
        s(n)\sim \dfrac{1}{s}\log n,\,M(n)=n.
    \end{align*}
    Let $w(t)=\log ^\gamma t(\gamma>0)$. Based on the asymptotic estimation of the Lambert $W$-function, we can obtain $F^{-1}(t)\sim \dfrac{t}{\log t}.$  By Theorem \ref{thm:5.1}, 
    we have 
    \[\int^\infty |w''(t)|M(F^{-1}(t))dt\sim \int^{\infty}\dfrac{1}{t\log^{2-\gamma}t}dt.\]
    Therefore, for $\gamma>0$, $E(\log^\gamma\tau)<\infty$ if and only if $\gamma<1$.

    When $0\le s-r<m$, we have 
     \begin{align*}
        s(n)&\sim \dfrac{1}{s}\sum_{k=1}^n k^\frac{r-s}{m}\sim n^{\frac{r-s+m}{m}},
        M(n)\sim n^{\frac{s-r}{m}}.
    \end{align*}
    Take $w(t)=t^a(0<a<1)$. Similarly, by Theorem \ref{thm:5.1}, we have 
     \[\int^\infty |w''(t)|M(F^{-1}(t))dt\sim \int^{\infty} t^{a-1+\frac{s-r-m}{m}} dt.\] Therefore, for $a>0$, $E(\tau^a)<\infty$ if and only if $a<1-\dfrac{s-r}{m}$.
    \end{proof}

\subsection{Bilateral birth-death process}
    For a bilateral birth-death process $\{X_t,t\geq 0\}$ on $\zz$, its density matrix $Q=(q_{ij})(i,j\in \zz)$ has the form 
    \[\left\{ \begin{aligned}
    	q_{ij}&=0,\qquad|i-j|>1,\\
    	q_{i,i-1}&=\alpha_i>0,\\
    	q_{i,i+1}&=\beta_i>0,\\
    	q_i:=&-q_{ii}=\alpha_i+\beta_i<\infty.
    	\end{aligned}
    	\right.\]
    The invariant measure $m_i$ has the following representation
    \[\left\{
    \begin{aligned}
    m_n&=\dfrac{\alpha_{-1}\alpha_{-2}\cdots \alpha_{n+1}}{\beta_0\beta_{-1}\cdots \beta_{n+1}\beta_n},\,\,n < -1,\\
    m_{-1}&=\dfrac{1}{\beta_0\beta_{-1}},\,m_{0}=\dfrac{1}{\alpha_0\beta_0},\,m_{1}=\dfrac{1}{\alpha_0\alpha_{1}},\\
    m_n&=\dfrac{\beta_1\beta_2\cdots \beta_{n-1}}{\alpha_0\alpha_1\alpha_2\cdots \alpha_{n-1}\alpha_n},\,\,n >1.
    \end{aligned}\right.\]

    In the microscopic description of chemical reactions, there is an important bilateral birth-death process with alternating rates. In the study of chain molecular diffusion, Stockmayer\cite{stockmayer1971local} models a molecule as a freely-joined chain of two regularly alternating kinds of atoms. The two kinds of
    atoms have alternating jump rates, and these rates are reversed for even and odd  labeled beads. This model has also attracted the attentions\cite{di2012bilateral}. We consider a birth-death process $\{X_t;t \ge 0\}$ with state space $\zz$, and for $n\in \zz$ denote by \[\beta_{2n}=\alpha_{2n+1}=r,\beta_{2n+1}=\alpha_{2n}=s,\] its birth and death rates. We call such a process a bilateral birth-death process with alternating rates $(r,s)$.

    \begin{thm}[Bilateral birth-death process with alternating rates]
    Consider a bilateral birth-death process with alternating rates $(r,s)$ on $\zz$, we have $E(\tau^a)<\infty$ if and only if $0<a<\dfrac{1}{2}$.
    \end{thm}

    \begin{proof}
        It is easy to obtain that the invariant measure $m_i=\dfrac{1}{rs}$ for $i\in \zz$. It follows that $A^+_{N}\sim A^-_{N}\sim \dfrac{1}{N},B_N^+(\lambda)\sim B_N^-(\lambda)\sim \lambda N.$ Take $N=\sqrt{\dfrac{1}{\lambda}}$, by Theorem \ref{thm:4.9}, we have $C_{0\leftrightarrow \Delta}(\lambda)\sim \sqrt{\lambda}.$ According to Theorem \ref{thm:4.4}, assume that $w(t)=t^a$ where $0<a<1$, we have 
        \[ \int^{\infty} t |w''(t)|C_{0\leftrightarrow\Delta}(1/t) dt\sim \int^{\infty} t^{a-3/2} dt.\]
        Hence for $0<a<1$, we have $E(\tau^a)<\infty$ if and only if $0<a<\dfrac{1}{2}$.
    \end{proof}

    At the end of this section, we present the construction of a classical problem. We know that there is no necessary connection between the positive recurrence of the CTMC and the positive recurrence of the corresponding embedded DTMC (Discrete Time Markov Chain). We give an example where the process is null recurrent but the embedded chain is positive recurrent, and give a characterization of the first returning speed.

    \begin{ex}\label{ex:5.5}
    Consider a bilateral birth-death process $\{X_t,t\geq 0\}$ with the state space $I=\zz$. For a positive parameter $\gamma>0$,     
    \[\left\{ \begin{aligned}
    	\alpha_i&=\dfrac{2}{3\gamma^i},\,\beta_i=\dfrac{1}{3\gamma^i},\,i>0,\\
    	\alpha_0&=\beta_0=\dfrac{1}{2},\\
    	\alpha_i&=\dfrac{1}{3\gamma^{-i}},\,\beta_i=\dfrac{2}{3\gamma^{-i}},\,i<0.\\
    	\end{aligned}
    \right.\]

    Denote the invariant measure of the embedded chain by $\pi(i)$, then $\pi(\cdot)$ satisfies  
    \[\left\{ \begin{aligned}
    	\dfrac{1}{2}\pi(0)&=\dfrac{2}{3}\pi(1)=\dfrac{2}{3}\pi(-1),\\
    	\dfrac{1}{3}\pi(i)&=\dfrac{2}{3}\pi(i+1),\quad i\ge 1,\\
    	\dfrac{1}{3}\pi(i)&=\dfrac{2}{3}\pi(i-1),\quad i\le -1.\\
    	\end{aligned}
    \right.\] By simple calculations, we know that its embedded chain has a stationary distribution
    \[\pi(i)=\dfrac{3}{2^{|i|+3}}\mbox{\,\,for\,\,} i\ne 0,\quad\,\pi(0)=\dfrac{1}{4},\]
    and the positive recurrence follows.

    For the invariant measure of $\{X_t,t\geq 0\}$, 
    \[m_n=\left\{
    \begin{aligned}
    &6\left(\frac{\gamma}{2}\right)^{|n|},&n \ne 0,\\
    &4, &n=0.\\
    \end{aligned}\right.\]
    That is to say, when \(\gamma \geq 2\), the process is null recurrent.
    \begin{enumerate}[(1)]
        \item If \(\gamma > 2\), $A_N^+\sim A_N^-\sim 2^{-N}$ and $B_N^+(\lambda)\sim B_N^-(\lambda)\sim \lambda (\gamma/2)^N$. Take $N=-\dfrac{\log \lambda}{\log \gamma}$, we have \[C_{0\leftrightarrow \Delta}(\lambda)\sim 2^{\log_\gamma \lambda}=\lambda^{\log_\gamma 2}.\]
        According to Theorem \ref{thm:4.4}, assume that $w(t)=t^a$ and $0<a<1$, 
        we have \[\int^{\infty} t |w''(t)|C_{0\leftrightarrow\Delta}(1/t) dt\sim \int^{\infty} t^{a-1-\log_\gamma2} dt.\]
        Hence, for $0<a<1$, $E(\tau^a)<\infty$ if and only if $0<a<\log_{\gamma}2$.
           
        \item If \(\gamma = 2\), $A_N^+\sim A_N^-\sim 2^{-N}$ and $B_N^+(\lambda)\sim B_N^-(\lambda)\sim \lambda N$. Take $N=\left[\dfrac{W(\frac{\log 2}{\lambda})}{\log 2}
        \right]\sim \log \dfrac{1}{\lambda}$, where $W(x)$ is Lambert W-function. Then we have 
        \[C_{0\leftrightarrow \Delta}(1/t)\sim \dfrac{\log t}{t}.\]
        Let $w(t)=t^a(0<a<1)$, according to Theorem \ref{thm:4.4},  
        we have  
        \[t |w''(t)|C_{0\leftrightarrow\Delta}(1/t)\sim \dfrac{\log t}{t^{2-a}}.\] 
         Hence $E(\tau^a)<\infty$ always holds for $0<a<1$.

        If we take $w(t)=\dfrac{t}{\log^b t}(b>0)$, then we have \[t |w''(t)|C_{0\leftrightarrow\Delta}(1/t)\sim \dfrac{1}{t\log^b t}.\]  Hence, for $b>0$, $E(w(\tau))<\infty$ always holds for $b>1$. 
        
    \end{enumerate}
    
\end{ex}

\bibliography{CTMCandEN}
\bibliographystyle{unsrt}

\end{document}